\documentclass[a4paper,oneside,reqno,11pt]{article}
\usepackage{latexsym,amsthm,amscd,amsmath,amssymb}
\usepackage[mathscr]{eucal}

\newtheorem{theorem}{Theorem}[section]
\newtheorem{lemma}[theorem]{Lemma}
\newtheorem{proposition}[theorem]{Proposition}
\newtheorem{corollary}[theorem]{Corollary}
\newtheorem{e-definition}[theorem]{Definition}
\newtheorem{remark}{\it Remark\/}

\def\bC{\mathbb C}

\def\bC{\mathbb C}
\def\cE{\mathcal E}
\def\cF{\mathcal F}

\def\n{\noindent}

\def\n{\noindent}

\def\ul#1{{\smash{\rm#1}}}
\def\SUPP{\mathop{\ul{SUPP}}}
\def\IM{\mathop{\ul{IM}}}
\def\IC{\mathop{\ul{IC}}}
\def\IT{\mathop{\ul{IT}}}

\def\PSH{\mathop{\rm PSH}\nolimits}

\begin{document}
\centerline{\large\bf Singularity invariants of plurisubharmonic functions and complex spaces}
\vskip 0.5cm
\n
\centerline{\bf Ph\d{a}m Ho\`ang Hi\d{\^e}p}
\vskip3mm
\centerline{ICRTM, Institute of Mathematics,}
\centerline{Vietnam Academy of Science and Technology,}
\centerline{18, Hoang Quoc Viet, Hanoi, Vietnam}
\vskip 1cm
\n
{\bf Abstract.} In this paper, we combine tools from pluripotential theory and commutative algebra to study singularity invariants of plurisubharmonic functions. We establish several relationships between the singularity invariants of plurisubharmonic functions and those of holomorphic functions. These results yield a sharp lower bound for the log canonical threshold of a plurisubharmonic function. Our bound simultaneously improves upon the main result of Demailly and Pham (Acta Math. 212: 1--9, 2014), the classical result of Skoda (Bull. Soc. Math. France 100: 353--408, 1972), and the lower estimate of T. de Fernex, L. Ein and M. Musta\c{t}\v{a} (Math. Res. Lett. 10: 219--236, 2003), which has played a crucial role in recent developments in birational geometry. Finally, we explore how singularity invariants associated with plurisubharmonic functions can be extended to complex spaces.
\vskip0.5cm
\n Keywords and Phrases:  Lelong number, Monge-Amp\`ere operator, log canonical threshold, Hilbert-Samuel multiplicity.

\n
\section{Introduction and main results:}

In mathematical classification problems, invariants have played a central role. The purpose of this paper is to study the singularities of plurisubharmonic functions and their applications in complex geometry. Throughout this article, we define 
$$d^c=\frac {i} {2\pi} (\overline {\partial} - \partial),$$
so that 
$$dd^c=\frac {i} {\pi}\partial \overline {\partial}.$$ 
This normalization of the $d^c$ operator is chosen so that we have the identity
$$(dd^c\log|z|)^n=\delta_0,$$
for the Monge-Amp\`{e}re operator in $\mathbb C^n$. 

The Monge-Amp\`{e}re operator is defined on locally bounded plurisubharmonic functions according to the definition of Bedford and Taylor (\cite{BT76}, \cite{BT82}). It can also be extended to plurisubharmonic functions with isolated or compactly supported poles, as shown in \cite{De93a}. It is worth noting that the results established by Demailly apply to more general currents, which have found significant applications in intersection theory.

If $\Omega\subset\mathbb C^n$ is an open subset, we denote by $\PSH(\Omega)$ (respectively, \ $\PSH^-(\Omega)$) the set of plurisubharmonic (respectively, negative plurisubharmonic functions) on $\Omega$.

\n
\begin{e-definition} Let $\Omega$ be a bounded hyperconvex domain $($i.e.\ a domain possessing a negative psh exhaustion$)$. Following Cegrell {\rm [Ce04]}, we introduce certain classes of psh functions on $\Omega$, in relation with the definition of the Monge-Amp\`ere operator:
$$\cE_0(\Omega)=\{\varphi\in \PSH^-(\Omega):\ \lim\limits_{z\to\partial\Omega}\varphi (z)=0,\ \int_\Omega (dd^c\varphi)^n<+\infty\},
\leqno{\rm(i)}$$
$$\cF(\Omega)=\{\varphi\in \PSH^-(\Omega):\ \exists\ \cE_0(\Omega)\ni\varphi_p\searrow\varphi,\ \sup\limits_{p\geq 1}\int_\Omega(dd^c\varphi_p)^n<+\infty\},\leqno{\rm(ii)}$$
$$\cE(\Omega)=\{\varphi\in \PSH^-(\Omega):\ \exists\ \varphi_K\in\cF(\Omega)\ \text{such that}\ \varphi_K=\varphi\ \text{on}\ K,\ \forall\ K\subset\subset\Omega\}.\leqno{\rm(iii)}$$
\end{e-definition}

It was proved in \cite{Ce04} that the class $\mathcal E(\Omega)$ is the largest subclass of $\PSH^-(\Omega)$ on which the Monge-Amp\`{e}re operator is well-defined. It is easy to see that $\mathcal E(\Omega)$ contains all plurisubharmonic functions that are locally bounded outside isolated singularities.

For $\varphi\in\PSH(\Omega)$ and $0\in\Omega$, Demailly and Koll\'{a}r (\cite{DK01}) introduce the log canonical threshold of $\varphi$ at point $0$, defined as:
$$c(\varphi) = \sup\big\{ c>0:\ e^{-2c\varphi} \text{ is } L^1 \text{ on a neighborhood of }0\big\},$$
For $1\leq k\leq n$, we define
$$c_k( \varphi ) = \sup\big\{c( \varphi |_{H} ):\ H \text{ is a $k$-dimensional complex subspace through }0\big\},$$
where $\varphi |_{H}$ denotes the restriction of $\varphi$ to the subspace $H$. 

For $\varphi\in\mathcal E(\Omega)$ and $1\leq k\leq n$, we define the intersection numbers
$$e_{k} (\varphi ) = \int_{\{0\}} (dd^c\varphi)^k\wedge (dd^c\log\Vert z\Vert)^{n-k}$$
which can also be interpreted as the Lelong numbers of the current $(dd^c\varphi)^k$ at the origin.

Our first main result is stated in the following theorem:

\begin{theorem}\label{Theorem1} Let $\varphi\in\cE(\Omega)$ and suppose $0\in\Omega$. Then $c_k(\varphi)=+\infty$ if $e_1(\varphi)=0$, and otherwise, for all $1\leq k\leq n$, we have 
$$c_k (\varphi) \geq c_{ k - 1} (\varphi ) + \frac { 1 } { \prod\limits_{ i = 1 }^{ k - 1 } \big( c_{ i } ( \varphi ) - c_{ i - 1 } ( \varphi ) \big) e_k (\varphi ) },$$
where we set $c_0 ( \varphi ): = 0$.
\end{theorem}

\begin{remark} The main theorems in \cite{DH14} and \cite{Hi18} are consequences of Theorem \ref{Theorem1}. Indeed, for all $1\leq i\leq n$, we define
$$x_i = c_i (\varphi ) - c_{ i - 1 } ( \varphi ),$$
and 
$$y_i = \frac { e_{ i - 1 } ( \varphi ) } { e_i ( \varphi ) },$$
where $c_0 ( \varphi ) = 0$ and $e_{ 0 } ( \varphi ) = 1$. By Theorem 1.3 in \cite{Hi17} and Lemma 2.1 in \cite{DH14}, we have
$$x_1\geq x_2\geq ... \geq x_n,$$
and
$$y_1\geq y_2 \geq ... \geq y_n.$$
On the other hand, by Theorem \ref{Theorem1}, for all $1\leq k\leq n$, we obtain
$$\prod\limits_{ i = 1 }^{ k } x_i \geq \prod\limits_{ i = 1 }^{ k } y_i ,$$
and therefore, for all $1\leq k\leq n$, we also have
$$\sum\limits_{ i = 1 }^{ k } x_i\geq \sum\limits_{ i = 1 }^{ k } y_i  .$$
\end{remark} 

For every non-negative Radon measure $\mu$ on a neighbourhood of $0\in\mathbb C^n$, we define the weighted log canonical threshold of $\varphi$ with weight $\mu$ at $0$ as:
$$c_{\mu}( \varphi ) = \sup\big\{ c \geq 0:\ e^{-2c\,\varphi} \hbox{ is } L^1(\mu ) \hbox{ on a neighborhood of }0\big\}\in [0,+\infty].$$
For an ideal $\mathcal I$ of $\mathcal O_{ \mathbb C^n, 0}$, the ring of germs of holomorphic functions at $0\in\mathbb C^n$, we denote by $\mathcal { D } ( \mathcal I )$ the ideal generated by 
$$\{ f, z_i \frac { \partial f } { \partial z_j }:\ 1\leq i,j\leq n;\ f\in\mathcal I\}.$$
We denote by $\bar{ \mathcal D } (\mathcal I)$ the maximal element of the set of ideals $\mathcal J$ such that $\mathcal { D } ( \mathcal J )\subset \mathcal { D } ( \mathcal I )$, which means that
$$\bar{ \mathcal {D} } (\mathcal I) = \sum\limits_{ \mathcal {D} ( \mathcal J )\subset \mathcal { D } ( \mathcal I ) } \mathcal J,$$

For a holomorphic $f$ on a neighbourhood of $0\in\mathbb C^n$, we denote $\varphi_{ f }$ and $\varphi_{ \bar { \mathcal  D } (f) }$ as follows:
$$\varphi_{ f } = \log |f|,$$
and
$$\varphi_{ \bar { \mathcal  D } (f) } = \frac { 1 } { 2 } \log \big( | f |^2 + \sum\limits_{j=1}^m |f_j|^2 \big ),$$
where $\bar { \mathcal  D } ( f ):= \bar { \mathcal  D } ( <f> )$ is the ideal generated by the set $\{ f, f_1,...,f_m \}$.

Our second main result is the following theorem:

\begin{theorem}\label{Theorem2} 
Let $f$ be a holomorphic function on a neighbouhood of $0\in\mathbb C^n$. Then

$$c_{ || z ||^{2t} dV_{2n} } ( \varphi_{ f } ) = \min \left\{ c_{ || z ||^{2t} dV_{2n} } ( \varphi_{ \bar { \mathcal  D } ( f ) } ), 1 \right\}, \text{ for all }\ t > -n.$$

\end{theorem}

Our main results imply the following corollaries, which are independently interesting.

\begin{corollary}\label{Corollary1} 
Let $f$ be a holomorphic function on a neighbouhood of $0\in\mathbb C^n$. Assume that the dimension of the analytic set $\{f = \frac { \partial f } { \partial z_1 } = ... = \frac { \partial f } { \partial z_n }= 0\}$ at $0$ is $k$, then

$$c_j ( \varphi_{ f } ) \geq \min \left\{ c_{ j - 1} ( \varphi_{ \bar { \mathcal  D } ( f ) } ) + \frac { 1 } { \prod\limits_{ i = 1 }^{ j - 1 } ( c_{ i } ( \varphi_{ \bar { \mathcal  D } ( f ) } ) - c_{ i - 1 } ( \varphi_{ \bar { \mathcal  D } ( f ) } ) ) e_j ( \varphi_{ \bar { \mathcal  D } ( f ) } ) }, 1\right\},$$

for all $1\leq j\leq n - k$.
\end{corollary}

\begin{corollary}\label{Corollary2} 
Let $f$ be a germ of a holomorphic function at $0 \in \mathbb{C}^n$. Assume that 
$$z_1^m, \dots, z_n^m \in \left< f, \frac{\partial f}{\partial z_1}, \dots, \frac{\partial f}{\partial z_n} \right>.$$
Then, we have
$$c_n(\varphi_f) \geq \min \left\{ c_{n-1}(\varphi_{ f } ) + \frac{1}{m+1}, 1 \right\}.$$
\end{corollary}

We refer to the following works for fundamental applications of singularities of plurisubharmonic functions in complex geometry and birational geometry: \cite{Bi14}, \cite{Be09}, \cite{BBN11}, \cite{BFFU15}, \cite{Ch05}, \cite{Co95}, \cite{Co00}, \cite{dFEM03}, \cite{dFEM04}, \cite{dFEM10}, \cite{DM23}, \cite{EM21}, \cite{FJ05} \cite{De93b}, \cite{DP04}, \cite{GZ15a}, \cite{GZ15b}, \cite{HMX14}, \cite{EGZ09}, \cite{IM72}, \cite{Pa07}, \cite{Pu87}, \cite{Pu02}, \cite{Is01}, \cite{PS00}, \cite{Ti87}.

\section{Some notations in pluripotential theory and commutative algebra:}

In this section, we provide basic notations and results from pluripotential theory and commutative algebra that are necessary for the subsequent sections. For an introduction to pluripotential theory, the books Complex Analytic and Differential Geometry (\cite{De12}) and Notions of Convexity (\cite{Ho94}) are recommended. For an introduction to commutative algebra, we recommend the monographs Introduction to Commutative Algebra (\cite{AM69}), Commutative Ring Theory (\cite{Ma86}), and Integral Closure of Ideals, Rings, and Modules (\cite{SH06}). Further information about mixed Hilbert-Samuel multiplicities can be found in \cite{AB58} and \cite{Te73}.

{\bf 2.1. Singularity equivalence on plurisubharmonic functions:}

We define singularity equivalence of plurisubharmonic functions on a compact set as follows:

\begin{e-definition}\label{singularity_equivalent} Two plurisubharmonic functions $\varphi$ and $\psi$  are said to be singularity equivalent on a compact set $K$ in $\Omega$ if there exists a neighborhood $U$ of $K$ such that 
$$\sup \{ | \varphi  ( z ) - \psi ( z ) | : z\in U \}< +\infty.$$
We denote by $PSH_{K}$ the set of equivalence classes of plurisubharmonic functions on a neighbourhood of the set $K$.
\end{e-definition}

{\bf 2.2. The weighted log canonical thresholds:}

For a Lebesgue measure function $\varphi$ and a Radon measure $\mu$ defined on a neighbourhood of $0\in\mathbb C^n$, we define the weighted log canonical threshold of $\varphi$ with respect to $\mu$ at $0$ as:
$$c_{\mu}( \varphi ) = \sup\big\{ c \geq 0:\ e^{-2c\,\varphi} \hbox{ is } L^1(\mu ) \hbox{ on a neighborhood of }0\big\}\in [0,+\infty].$$

We list below some basic properties of the weighted log canonical threshold, which will be useful in later sections.

\begin{proposition}\label{weighted_lct}
Let $\varphi_1$, ... , $\varphi_k$ be Lebesgue measurable functions, and let $\mu$ be a Radon measure on the polydisc $\Delta^n\subset\mathbb C^n$. Then the following properties hold:

\n
i) $\frac { 1 } { c_{ \mu } ( \sum\limits_{ i = 1 }^k \varphi_i ) } \leq \sum\limits_{ i = 1 }^k \frac { 1 } { c_{\mu} ( \varphi_i ) }$,

\n
ii) $c_{ \mu } ( \min\limits_{ 1\leq i\leq k }  \varphi_i ) = \min\limits_{ 1\leq i\leq k }  c_{ \mu } ( \varphi_i )$.
\end{proposition}

\begin{proof} 
i) This follows directly from H\"{o}lder inequality.

\n
ii) This follows from the elementary inequalities:
$$\max\limits_{ 1\leq i\leq k } e^{ -2c \varphi_i }\leq e^{ -2c \min\limits_{ 1\leq i\leq k } \varphi_i }\leq\sum\limits_{ i = 1 }^k e^{ -2c \varphi_i },\ \forall\ c > 0.$$
\end{proof}

\begin{proposition}\label{concave_property} Let $\varphi$ be a Lebesgue measurable function on a domain $\Omega\subset\mathbb C^n$, with $0\in\Omega$. Then the map
$$(-n, +\infty )\ni t\to c_{ ||z||^{ 2 t } dV_{2n} } (\varphi)$$ 
is concave and increasing.
\end{proposition}

\begin{proof} 
It follows directly from H\"{o}lder inequality.
\end{proof}

 \begin{proposition}\label{weightedlct_inequality} Let $\varphi\in\PSH(\Omega)$ and $0\in\Omega$ be such that 
$$\varphi (z) \geq A \log || z || - B,$$
for some constants $A, B >0$. Then, for all $t\geq s > -n$, we have 
$$c_{ || z ||^{ 2t } dV_{ 2n } } ( \varphi ) \geq c_{ || z ||^{ 2s } dV_{ 2n } } ( \varphi ) + \frac { t - s } { A}.$$ 
\end{proposition}

\begin{proof} 
Let $c < c_{\|z\|^{2s} dV_{2n}} (\varphi) $. We aim to show that
$$c_{\|z\|^{2t} dV_{2n}} (\varphi) \geq c + \frac{t - s}{A}.$$
To prove this, note that since $c < c_{\|z\|^{2s} dV_{2n}} (\varphi)$, we have the following integral condition for some neighborhood $U$ of $0 \in \mathbb{C}^n$:
$$\int\limits_U e^{-2c\varphi} \|z\|^{2s} dV_{2n} < +\infty.$$
Additionally, since $\varphi(z) \geq A \log \|z\| - B$, we obtain the inequality:
$$\int\limits_U e^{-2\left(c + \frac{t - s}{A}\right) \varphi} \|z\|^{2t} dV_{2n} \leq e^{\frac{2B(t - s)}{A}} \int\limits_U e^{-2c\varphi} \|z\|^{2s} dV_{2n} < +\infty.$$
This implies that
$$c_{\|z\|^{2t} dV_{2n}} (\varphi) \geq c + \frac{t - s}{A}.$$
\end{proof}

Relationships between the weighted log canonical threshold $c_{ ||z||^{2 t} dV_{2n} } (\varphi)$ and the log canonical threshold of the restriction of $\varphi$ to the subspace of  dimension $k$ through the origin (see \cite{La13}, \cite{Hi17}) are given in the following result:

\begin{proposition}\label{weighted_property} Let $\varphi\in\PSH(\Omega)$ and $0\in\Omega$. Then, for all $1\leq k\leq n$,
$$c_{k} (\varphi ) = c_{ ||z||^{ 2 ( k - n ) } dV_{2n} } (\varphi).$$ 
\end{proposition}

In fact, by combining Propositions \ref{weightedlct_inequality} and \ref{weighted_property}, we obtain the following inequality:

\begin{proposition}\label{ltc_inequality} Let $\varphi\in\PSH(\Omega)$ with $0\in\Omega$, and suppose that
$$\varphi (z) \geq A \log || z || - B,$$ 
for some constants $A, B >0$. Then
$$c_{ n } ( \varphi ) \geq c_{ n-1 } ( \varphi )  + \frac { 1 } { A}.$$ 
\end{proposition}

{\bf 2.3. Mixed Monge-Amp\`{e}re masses:}

Let $\varphi_1$, ..., $\varphi_n$ be plurisubharmonic functions in the Cegrell's class $\mathcal E$ on a neighborhood of $0\in\mathbb C^{ n }$. We denote the mixed Monge-Amp\`{e}re mass of $\varphi_1$, ..., $\varphi_n$ at the origin as:
$$MA_{ 0 } ( \varphi_1 , ..., \varphi_n ) = \int\limits_{ \{ 0 \} } dd^c \varphi_1 \wedge ... \wedge dd^c \varphi_n.$$
This number is a singularity invariant of $\varphi_1$, ..., $\varphi_n$. Therefore, the mixed Monge-Amp\`{e}re mass can be defined for elements in $PSH_{ 0 }$, the set of equivalence classes of plurisubharmonic functions near the origin:

\begin{e-definition}\label{mix_Monge_Ampere_mass}
The mixed Monge-Amp\`{e}re mass of $\varphi_1, ...,\varphi_n\in \PSH_0$ is given by
$$MA_{ 0 } ( \varphi_1 , ..., \varphi_n ) = \lim\limits_{ j\to +\infty } \int\limits_{ \{ 0 \} } dd^c \max ( \varphi_1 , j \log || z || ) \wedge ... \wedge dd^c \max ( \varphi_n , j \log || z || ).$$
To simplify notations, we define
$$e_k ( \varphi ) = \lim\limits_{ j\to +\infty } \int\limits_{ \{ 0 \} } ( dd^c \max ( \varphi, j \log || z || ) )^{ k } \wedge ( dd^c \log || z || )^{ n-k }.$$
\end{e-definition}

{\bf 2.4. Mixed Hilbert-Samuel multiplicity:}

Let $A$ be a Noetherian local ring with maximal ideal $\mathfrak{m}_A$ and $M$ be a finitely generated $A$-module of dimension $n$. Assume that $\mathcal I_1$, ..., $\mathcal I_n$ are $\mathfrak{m}_A$-primary ideals in $A$. The mixed Hilbert-Samuel multiplicity of $M$ with respect to $\mathcal I_1$, ..., $\mathcal I_n$ is defined as the coefficient of the monomial $x_1...x_n$ in the Hilbert polynomial
$$P(x_1,...,x_n) = \text{Length}_A \big( \frac { M } { \mathcal I_1^{ x_1 } ... \mathcal I_n^{ x_n } M } \big),$$
and is denoted by $e_M ( \mathcal I_1, ..., \mathcal I_n )$.

In the case of $\mathcal I_1=... =\mathcal I_n$, we write
$$e_M ( \mathcal I ):= e_M ( \mathcal I, ..., \mathcal I ) = \lim\limits_{ k\to +\infty } \frac { n! } { k^n } \text{ Length}_A \big( \frac { M } { \mathcal I^{ k} M } \big).$$
For convenience of notation, in the case of $M=A$, we denote
$$e ( \mathcal I_1, ..., \mathcal I_n ): = e_A ( \mathcal I_1, ..., \mathcal I_n ),$$
and
$$e_k ( \mathcal I ): = e_A(\underbrace{\mathcal{I}, \dots, \mathcal{I}}_{k}, \underbrace{\mathfrak{m}_A, \dots, \mathfrak{m}_A}_{n-k}),$$
where $\mathcal I$ (resp. $\mathfrak{m}_A$) appear $k$ (resp. $n-k$) times.

We state the following lemma, which is needed to compute the length of modules:

\begin{lemma}\label{length_dimension}
Let $(A,\mathfrak{m}_A)$ be a local ring with maximal ideal $\mathfrak{m}_A$ and $M$ be an $A$-module. Assume that $A = K \oplus \mathfrak{m}_A$ where $K$ is the residue field. Then
$$\text{Length}_A (M) = \dim_K (M).$$
\end{lemma}

\begin{proof}
We begin by noting that
$$\text{Length}_A(M) \leq \dim_K(M).$$
Our goal is to establish the reverse inequality:
$$\dim_K(M) \leq \text{Length}_A(M).$$
Without loss of generality, assume that $\text{Length}_A(M) = n < + \infty$. Then, there exists a chain of submodules:
$$0 = M_0 \subset M_1 \subset \cdots \subset M_n = M,$$
such that each successive quotient $M_j / M_{j-1}$ is a simple $A$-module for $ 1 \leq j \leq n $. Given that $A = K \oplus \mathfrak{m}_A$, it follows that each simple $A$-module is isomorphic to the residue field $K$. Therefore,
$$M_j / M_{j-1} \cong K \quad \text{for all } 1 \leq j \leq n.$$
Consequently,
$$\dim_K(M) = \sum_{j=1}^n \dim_K(M_j / M_{j-1}) = \sum_{j=1}^n 1 = n.$$
Thus, $\dim_K(M) = \text{Length}_A(M) $.
\end{proof}

Let $\mathcal O_{ \mathbb C^n, 0}$ be a ring of germs of holomorphic functions at $0\in\mathbb C^n$. This ring is a regular local ring of dimension $n$ with maximal ideal  $\mathfrak{m}_{\mathcal{O}_{\mathbb{C}^n, 0}}$ generated by $\{z_1, \ldots, z_n\}$. Lemma \ref{length_dimension} directly leads to the following consequence:

\begin{proposition}\label{length_dimension1}
Let $\mathcal O_{ \mathbb C^n, 0}$ denote the ring of germs of holomorphic functions at $0\in\mathbb C^n$ and $M$ be an $\mathcal O_{ \mathbb C^n, 0}$-module. Then
$$\text{Length}_{\mathcal O_{ \mathbb C^n, 0}} (M) = \dim_{\mathbb C} (M).$$
\end{proposition}

For any $\mathfrak{m}_{ \mathcal O_{ \mathbb C^n, 0} }$-primary ideal $\mathcal I\subset\mathcal O_{ \mathbb C^n, 0}$, we have
$$\aligned e_n ( \mathcal I ) &= \lim\limits_{ k\to +\infty } \frac { n! } { k^n } \text { Length}_{ \mathcal O_{ \mathbb C^n, 0} } \big( \frac {\mathcal O_{ \mathbb C^n, 0} } { \mathcal I^{ k} } \big)\\
&= \lim\limits_{ k\to +\infty } \frac { n! } { k^n } \dim_{ \mathbb C } \big(\frac {\mathcal O_{ \mathbb C^n, 0} } { \mathcal I^{ k} } \big).
\endaligned$$

{\bf 2.5. Monomial ordering on $\mathcal O_{ \mathbb C^n, 0}$:}

Following Robbiano's theorem (see \cite{Ro85}), every monomial ordering $<$ can be represented by an $n \times n$ matrix $M$ of rank $n$ with non-negative coefficients, such that for monomials $z^{\alpha}, z^{\beta}$:
$$z^{\alpha} < z^{\beta} \iff \alpha M < \beta M$$
in the lexicographic order.

We now state a well-known result (see \cite{Ei95}), which is essential in constructing a flat family of ideals:

\begin{proposition}\label{flat_family}
	Let $<$ be a monomial ordering on $\mathcal{O}_{\mathbb{C}^n, 0}$. Assume that $z^{\alpha^1}, \dots, z^{\alpha^k}$ are monomials in $\mathcal{O}_{\mathbb{C}^n, 0}$. Then there exists $a = (a_1, \dots, a_n) \in \mathbb{N}^n$ such that
	$$
	\inf\left\{ \sum_{i=1}^n a_i \beta_i : \beta \in \mathbb{N}^n,\ z^{\beta} > z^{\alpha^j} \right\}
	> \sum_{i=1}^n a_i \alpha_i^j,\quad \text{for all } 1 \leq j \leq k.
	$$
\end{proposition}

\begin{proof}
	By Robbiano's theorem, we may reduce to the case where $<$ is the lexicographic order.
	
	For each $1 \leq j \leq k$, let $\mathcal{I}_j$ denote the ideal generated by monomials $\{ z^\beta : z^\beta > z^{\alpha^j} \}$. Choose a finite basis $\{ z^{\beta^{j1}}, \dots, z^{\beta^{jp_j}} \}$ for $\mathcal{I}_j$.
	
	Since $\beta^{jl} > \alpha^j$ in the lex order, we can choose $a = (a_1, \dots, a_n) \in \mathbb{N}^n$ such that
	$$
	\min \left\{ \sum_{i=1}^n a_i \beta_i^{jl} : 1 \leq l \leq p_j \right\}
	> \sum_{i=1}^n a_i \alpha_i^j,
	\quad \text{for all } 1 \leq j \leq k.
	$$
	
	Therefore, for each $j$,
	$$
	\inf\left\{ \sum_{i=1}^n a_i \beta_i : \beta \in \mathbb{N}^n,\ z^\beta > z^{\alpha^j} \right\}
	> \sum_{i=1}^n a_i \alpha_i^j.
	$$
\end{proof}

{\bf 2.6. Hirokana's division theorem:}

We equip the ring ${\mathcal O}_{\mathbb C^n,0}$ of germs of holomorphic functions at $0$ with a monomial ordering $<$. For each $f(z)=a_{\alpha^1} z^{\alpha^1}+a_{\alpha^2} z^{\alpha^2}+\ldots\;$ with $a_{\alpha^j}\not = 0$, $j\geq 1$ and $z^{ \alpha^1 } < z^{ \alpha^2 } < \ldots\;$, we define the \emph{initial coefficient}, \emph{initial monomial}, and \emph{initial term} of $f$ to be, respectively $\IC(f): = a_{\alpha^1}$, $\IM(f): = z^{\alpha^1}$, $\IT(f): = a_{\alpha^1} z^{\alpha^1}$, and the support of $f$ to be $\SUPP(f): = \{z^{ \alpha^1 }, z^{ \alpha^2 }, \ldots\}$. For any ideal $\mathcal I$ of ${\mathcal O}_{\mathbb C^n,0}$, we define $\IM(\mathcal{I})$ to be the ideal generated by the set $\{ \IM(f) : f \in \mathcal{I} \}$. 

First, we recall the Division Theorem of Hironaka:

{\bf Division theorem of Hironaka} (see \cite{Ga79}, \cite{Ba82}, \cite{BM87}, \cite{Ei95}). {\it Let $f, g_1,\ldots,g_k\in {\mathcal O}_{\mathbb C^n,0}$. Then there exist $h_1,\ldots,h_k,s\in {\mathcal O}_{\mathbb C^n,0}$ such that
$$f=h_1g_1+\ldots+h_kg_k+s,$$
and $\SUPP(s)\cap \langle\IM(g_1),\ldots,\IM(g_k)\rangle =\emptyset$, where $\langle\IM(g_1),\ldots,\IM(g_k)\rangle$ denotes the ideal generated by the family $(\IM(g_1),\ldots,\IM(g_k))$.}

Hironaka's Division Theorem implies the following well-known result:

\begin{proposition}\label{dimension_ideal} Let $<$ be a monomial ordering on $\mathcal O_{ \mathbb C^n, 0}$. Assume that $\mathcal I$ is an $\mathfrak{m}_{ \mathcal O_{ \mathbb C^n, 0} }$-primary ideal of $\mathcal O_{ \mathbb C^n, 0}$. Then
$$\dim_{ \mathbb C } \big( \frac {\mathcal O_{ \mathbb C^n, 0} } { \mathcal I } \big) = \dim_{ \mathbb C } \big( \frac {\mathcal O_{ \mathbb C^n, 0} } { IM ( \mathcal I ) } \big).$$
\end{proposition}

\section{The weighted log canonical thresholds of toric plurisubharmonic functions:}

In this section, we provide explicit formulas for computing the weighted log canonical thresholds of toric plurisubharmonic functions $\varphi$ on the polydisc $\Delta^n$. A function $\varphi$ is called \emph{toric} if it depends only on the moduli of the coordinates, i.e., $\varphi(z) = \varphi(|z_1|, \ldots, |z_n|)$.

We consider the \emph{refined Lelong numbers} of $\varphi$ at $0$, introduced by Kiselman (see \cite{De93a}, \cite{Ki94}):
$$
\nu_{\varphi}(x) = \lim_{t \to -\infty} \frac{ \max\{ \varphi(z) : |z_1| = e^{t x_1}, \ldots, |z_n| = e^{t x_n} \} }{t},
$$
for $x = (x_1, \ldots, x_n) \in (\mathbb{R}^+)^n$.

We define the function $h_\varphi$ by
$$
\begin{aligned}
h_{\varphi}(x) &= \min\left\{ \frac{1}{\nu_{\varphi}(a)} : \sum_{i=1}^n a_i x_i = 1 \right\} \\
&= \min_{a \in S} \frac{1}{\nu_{\varphi} \left( \frac{a_1}{x_1}, \ldots, \frac{a_n}{x_n} \right)}\\
&= \inf \{ \sum\limits_{i=1}^n x_i y_i: \  \nu_{\varphi} ( y ) = 1\}\\
&= \inf \{ \sum\limits_{i=1}^n x_i y_i: \  \nu_{\varphi} ( y ) \geq 1\},
\end{aligned}
$$
where $S = \left\{ a \in (\mathbb{R}^+)^n : \sum_{i=1}^n a_i = 1 \right\}$.

The function $h_\varphi$ is concave and non-decreasing in each variable, and it satisfies the homogeneity property
$$
h_{\varphi}(t x) = t h_{\varphi}(x), \quad \forall x \in (\mathbb{R}^+)^n, \ t \in \mathbb{R}^+.
$$

By Neumann's Minimax Theorem (see \cite{Ne28}), we also have the dual formula:
$$
\begin{aligned}
\nu_{\varphi}(x) &= \min \left\{ \frac{1}{h_\varphi(a)} : \sum_{i=1}^n a_i x_i = 1 \right\} \\
&= \min_{a \in S} \frac{1}{h_\varphi\left( \frac{a_1}{x_1}, \ldots, \frac{a_n}{x_n} \right)}\\
&= \inf \{ \sum\limits_{i=1}^n x_i y_i: \  h_{\varphi} ( y ) = 1\}\\
&= \inf \{ \sum\limits_{i=1}^n x_i y_i: \  h_{\varphi} ( y ) \geq 1\}.
\end{aligned}
$$

By direct computations as in Theorem 5.8 of \cite{Ki94}, we obtain the following formula for the weighted log canonical threshold of a toric plurisubharmonic function $\varphi$:

\begin{equation}\label{eq:toric_lct_positive}
	c_{ ||z||^{ 2t } dV_{ 2n } } ( \varphi ) = \min \big\{ h_{ \varphi } ( 1 + t,1,\ldots,1 ),\ \ldots,\ h_{ \varphi } ( 1,\ldots,1,1 + t ) \big\},\quad \forall\ t \geq 0.
\end{equation}

In the case where $-n < t < 0$, we will use Kiselman's theorem and Simon's min-max theorem to derive the following explicit formulas for $c_{ ||z||^{ 2t } dV_{ 2n } } ( \varphi )$:

\begin{proposition}\label{toric_formula1} Let $\varphi$ be a toric plurisubharmonic function on $\Delta^n$ and $-n < t < 0$. Then 

\n
i) $c_{ ||z||^{ 2t } dV_{ 2n } } ( \varphi ) = \min\limits_{ x\in S } \frac { 1 + t \min\limits_{ 1\leq i \leq n } x_i } { \nu_{ \varphi } ( x ) },$

\n
ii) $c_{ ||z||^{ 2t } dV_{ 2n } } ( \varphi ) = \max\limits_{ x\in S } h_{ \varphi }  ( x ) \min \{ t+n , \frac { 1 } { \max\limits_{1\leq i \leq n} x_i }  \}.$
\end{proposition}

\begin{proof}
	\textbf{i)} By Kiselman's theorem (see Theorem 5.8 in \cite{Ki94}), we have
	$$
	\max_{x \in S} \left[ c\, \nu_{\varphi}(x) - t \min_{1 \leq i \leq n} x_i \right] < 1
	\quad \Longleftrightarrow \quad c < c_{ ||z||^{2t} dV_{2n} } (\varphi).
	$$
	This implies the identity
	$$
	c_{ ||z||^{2t} dV_{2n} } (\varphi) = \min_{x \in S} \frac{1 + t \min_{1 \leq i \leq n} x_i}{\nu_{\varphi}(x)}.
	$$
	
	\textbf{ii)} From part (i), we derive:
	$$
	\begin{aligned}
		c_{ ||z||^{2t} dV_{2n} } (\varphi)
		&= \min_{x \in S} \max_{y \in S} \left[1 + t \min_{1 \leq i \leq n} x_i\right] 
		h_{\varphi} \left( \frac{y_1}{x_1}, \ldots, \frac{y_n}{x_n} \right) \\
		&= \min_{x \in S} \max_{y \in (\mathbb{R}^+)^n} \frac{1 + t \min_{1 \leq i \leq n} x_i}{\sum_{i=1}^n y_i}
		h_{\varphi} \left( \frac{y_1}{x_1}, \ldots, \frac{y_n}{x_n} \right) \\
		&\geq \min_{x \in S} \frac{1 + t \min_{1 \leq i \leq n} x_i}{\sum_{i=1}^n a_i x_i} h_{\varphi}(a_1, \ldots, a_n) \\
		&= h_{\varphi}(a) \cdot \min \left\{ t+n,\ \frac{1}{\max_{1 \leq i \leq n} a_i} \right\},
	\end{aligned}
	$$
	for all $ a \in (\mathbb{R}^+)^n $. Therefore,
	$$
	c_{ ||z||^{2t} dV_{2n} } (\varphi) \geq \max_{x \in S} h_{\varphi}(x) \cdot \min \left\{ t+n,\ \frac{1}{\max_{1 \leq i \leq n} x_i} \right\}.
	$$
	
	Now we prove the reverse inequality. It suffices to show:
	$$
	c_{ ||z||^{2t} dV_{2n} } (\varphi) \leq \max_{x \in S} h_{\varphi}(x) \cdot \min \left\{ t+n,\ \frac{1}{\max_{1 \leq i \leq n} x_i} \right\}.
	$$
	Indeed, we observe:
	$$
	\begin{aligned}
		c_{ ||z||^{2t} dV_{2n} } (\varphi)
		&= \min_{x \in S} \max_{y \in S} \left[1 + t \min_{1 \leq i \leq n} x_i\right] 
		h_{\varphi} \left( \frac{y_1}{x_1}, \ldots, \frac{y_n}{x_n} \right) \\
		&\leq \min_{z \in S} \max_{y \in S} \min_{x \in S}
		\frac{1 + t \min_{1 \leq i \leq n} x_i}{\min_{1 \leq i \leq n} x_i z_i^{-1}} 
		h_{\varphi} \left( \frac{y_1}{z_1}, \ldots, \frac{y_n}{z_n} \right) \\
		&= \min_{z \in S} \min_{x \in S} \frac{1 + t \min_{i} x_i}{\min_{i} x_i z_i^{-1}} 
		\max_{y \in S} h_{\varphi} \left( \frac{y_1}{z_1}, \ldots, \frac{y_n}{z_n} \right) \\
		&= \min_{z \in S} \max_{x \in S} \sum_{i=1}^n x_i z_i 
		\cdot \min \left\{ t+n,\ \frac{1}{\max_{1 \leq i \leq n} x_i} \right\} \cdot \frac{1}{\nu_{\varphi}(z)} \\
		&= \min_{z \in S} \max_{x \in S} f(x, z),
	\end{aligned}
	$$
	where we define:
	$$
	f(x, z) = \sum_{i=1}^n x_i z_i \cdot \min \left\{ t+n,\ \frac{1}{\max_i x_i} \right\} \cdot \frac{1}{\nu_{\varphi}(z)}.
	$$
	
	Note that $ f(x, z) $ is quasi-concave-convex and upper semi-continuous-lower semi-continuous on the compact convex set $ S \times S $. By Sion's minmax theorem (see Theorem 3.4 in \cite{Si58}), we conclude:
	$$
	\begin{aligned}
		c_{ ||z||^{2t} dV_{2n} } (\varphi)
		&\leq \max_{x \in S} \min_{z \in S} f(x, z) \\
		&= \max_{x \in S} h_{\varphi}(x) \cdot \min \left\{ t+n,\ \frac{1}{\max_{1 \leq i \leq n} x_i} \right\},
	\end{aligned}
	$$
	as desired.
\end{proof}

Proposition \ref{toric_formula1} and Proposition \ref{weighted_property} imply the following consequence:

\begin{proposition}\label{toric_formula2} Let $\varphi$ be a toric plurisubharmonic function on $\Delta^n$ and $1\leq k\leq n$. Then 

\n
i) $c_{ k } ( \varphi ) = \min\limits_{ x\in S } \frac { 1 + ( k -n ) \min\limits_{ 1\leq i \leq n } x_i } { \nu_{ \varphi } ( x ) },$

\n
ii) $c_{ k } ( \varphi ) = \max\limits_{ x\in S } h_{ \varphi }  ( x ) \min \{ k , \frac { 1 } { \max\limits_{1\leq i \leq n} x_i } \}.$
\end{proposition}

\section{Mixed Monge-Amp\`{e}re mass and Hilbert-Samuel multiplicity:}

For each ideal $\mathcal I$ generated by $\{ f_1,...,f_m\}$ in ${\mathcal O}_{\mathbb C^n,0}$, we defined a singularity equivalence class of plurisubharmonic functions on a neighborhood of $0$ as follows:
$$\varphi_{\mathcal I}: = \frac 1 2 \log \sum_{j=1}^m | f_j |^2.$$
We denote by $\bar{\mathcal{I}}$ the integral closure of the ideal $\mathcal{I}$. According to an application of Skoda's theorem in \cite{Sk78} (see Corollary 10.5 in \cite{De12}), we have:
$$\varphi_{\bar { \mathcal I } }= \varphi_{\mathcal I}.$$
We also make the following simple observations:

\begin{proposition}\label{pro1} Let $\mathcal{I}$ be an ideal in the ring ${\mathcal{O}}_{\mathbb{C}^n, 0}$. Then the following four statements are equivalent:
	
	i) The function $\varphi_{\mathcal{I}}$ is locally bounded outside the origin.
	
	ii) $\varphi_{\mathcal{I}}$ belongs to Cegrell's class $\mathcal{E}$.
	
	iii) The ideal $\mathcal{I}$ is $\mathfrak{m}_{\mathcal{O}_{\mathbb{C}^n, 0}}$-primary.
	
	iv) The singularity invariant $e_n(\varphi_{\mathcal{I}})$ is finite; that is, $e_n(\varphi_{\mathcal{I}}) < +\infty$.
	
\end{proposition}

\begin{proof}
Obviously, the three statements $ i) $, $ ii) $, and $ iii) $ are equivalent, and the statement $ ii) $ implies $ iv) $. We now prove by contradiction that $ iv) $ implies $ iii) $. Assume that $ \mathcal{I} $ is not $\mathfrak{m}_{ \mathcal O_{ \mathbb C^n, 0} }$-primary. This implies that the dimension of the analytic germ $V(\mathcal{I})=\bigcap\limits_{ f\in\mathcal{I} } \{f = 0\}$ at $0$ is positive. We can choose a complex curve $ \gamma(t) = (\gamma_1(t), \dots, \gamma_n(t)) $ such that
$$0 \in \gamma \subset V(\mathcal{I}).$$
Let
$$\mathcal J := \{ f \in {\mathcal{O}}_{\mathbb{C}^n, 0} : f|_{\gamma} \equiv 0 \}.$$
We have the inclusion
$$\mathcal{I} \subset \operatorname{Rad}(\mathcal{I}) = \{f\in{\mathcal{O}}_{\mathbb{C}^n, 0}:  f|_{ V (\mathcal {I} ) } \equiv 0\}\subset \mathcal J.$$
Without loss of generality, we can assume that
$$\operatorname{ord}_0(\gamma_1) \geq \dots \geq \operatorname{ord}_0(\gamma_n).$$
Next, we have, for all $k\geq 1$
$$e_n(\mathcal J) \geq e_n \left( \mathcal J + < z_n^k > \right) \geq k  \nu ( [ \gamma ], 0),$$
where $\nu ( [ \gamma ], 0) = \int\limits_{0} [ \gamma ]\wedge (dd^c \log || z ||)^{n-1}$ is the Lelong number of current $[ \gamma ]$ at $0$ (see \cite{Th67} and also Theorem 7.7 in \cite{De12}). Taking the limit as $ k \to \infty $, we obtain
$$e_n(\mathcal J) = +\infty.$$
Moreover, since $ e_n(\mathcal{I}) \geq e_n(\mathcal J) $, it follows that
$$e_n(\mathcal{I}) = +\infty.$$
\end{proof}

\begin{proposition}\label{pro2} 
Let $\mathcal{I}$ be an ideal of the ring ${\mathcal{O}}_{\mathbb{C}^n, 0}$. Then, for $2 \leq k \leq n$, the following statements are equivalent:

i) The Krull dimension of the quotient ring $\big( \frac{{\mathcal{O}}_{\mathbb{C}^n, 0}}{\mathcal{I}} \big)$ is less than or equal to $n - k$.

ii) The singularity invariant $e_k(\varphi_{\mathcal{I}})$ is finite; that is, $e_k(\varphi_{\mathcal{I}}) < +\infty$.

\end{proposition}

\begin{proof}
\textbf{i) $\Rightarrow$ ii)}. Since the dimension of the analytic germ $V(\mathcal{I})=\bigcap\limits_{ f\in\mathcal{I} } \{f = 0\}$ at
at $0$ is less than or equal to $n - k$, we can choose a $k$-dimensional linear subspace $H$ of $\mathbb{C}^n$ such that the dimension of $V(\mathcal{I}) \cap H$ at $0$ is less than or equal to $0$. By Siu's theorem (see \cite{Si74} and also Theorem 5.14 in \cite{De92}), we have
$$e_k( \varphi_{ \mathcal{I} } ) \leq e_k(\varphi_{ \mathcal{I} }|_H) < +\infty.$$

\textbf{ii) $\Rightarrow$ i)}. By Siu's theorem (see \cite{Si74} and also Theorem 5.14 in \cite{De92}), we have, there exists a $k$-dimensional linear subspace $H$ of $\mathbb C^n$ such that
$$e_k(\varphi_{ \mathcal{I} }|_H) = e_k(\varphi_{ \mathcal{I} }) < +\infty.$$
Moreover, by Proposition \ref{pro1}, we obtain that the dimension of $V(\mathcal{I}) \cap H$ at $0$ is less than or equal to $0$. This implies that the dimension of $V(\mathcal{I})$ at $0$ is less than or equal to $n - k$.
\end{proof}

The relationship between the mixed Monge-Amp\`{e}re mass and the mixed Hilbert-Samuel multiplicity is given by the following proposition due to Demailly in \cite{De09} (see also \cite{Ki21}):

\begin{proposition}\label{Monge-Ampere_Hilbert_Samuel_1} Let $\mathcal I_1$, ..., $\mathcal I_n$ be ideals in $\mathcal O_{ \mathbb C^n, 0}$. Then
$$\int\limits_{ \{ 0\} } dd^c \varphi_{\mathcal I_1}\wedge ... \wedge dd^c \varphi_{\mathcal I_n} = e (\mathcal I_1, ..., \mathcal I_n).$$
\end{proposition}

\begin{proof} We will give a detail proof. Without of loss generality, we can assume that the ideals $\mathcal I_1$, ..., $\mathcal I_n$ are $\mathfrak{m}_{ \mathcal O_{ \mathbb C^n, 0} }$-primary. First, we consider the case where $\mathcal I = \mathcal I_1 = ... = \mathcal I_n$. By Lemma 10.3 in \cite{De12} , we can find an ideal $\mathcal J$ generated by $f_1,...,f_n$ in $\mathcal O_{ \mathbb C^n, 0}$ such that $\mathcal J\subset\mathcal I$ and $\bar { \mathcal J } = \bar { \mathcal I }$. We set $F = (f_1,...,f_n)$. 
	
We can find two neighborhoods $U$, $V$ of $0$ in $\mathbb C^n$ such that $F:U\to V$ is a proper holomorphic mapping. Remmert's proper mapping theorem (see \cite{BN90}) implies that $F( \{ \det\ \text{Jab} F = 0 \} )$ is a complex hypersurface in $U$. Thus, we can choose a sequence $\{w_j\}_{ \{ j\geq 1 \} }\subset V\backslash F( \{ \det\ \text{Jab} F = 0 \} )$ such that $w_j\to 0$ as $j\to +\infty$. 

We set
$$F_j = (f_1,...,f_n) - w_j,$$
and
$$\varphi_j = \log || F_j ||.$$
We have that $\varphi_j$ converges uniformly to $\varphi_{ \mathcal J }$ on $U\backslash \mathbb B (0,r)$ for all $r >0$. Moreover, by Theorem 3.4 in \cite{Hi08} , we get that $(dd^c \varphi_j)^n$ converges weakly to $( dd^c \varphi_{ \mathcal J } )^n$. On the other hand, we have
$$( dd^c \varphi_{ \mathcal J } )^n = \int\limits_{ \{ 0 \} } ( dd^c \varphi_{ \mathcal J } )^n \delta_{ \{ 0 \} },$$
and
$$(dd^c \varphi_j)^n = \sum\limits_{ 1\leq k\leq p } \delta_{ \{ z_{jk} \} },$$
where $\{ F_j = 0 \} = \{ z_{j1}, ..., z_{jp} \}$, $p=\text { mult}_0 ( F )$ is the multiplicity of $F$ at $0$, and $\delta_{ \{ z \} }$ is the Dirac measure at $z$. Therefore, we obtain
$$\int\limits_{ \{ 0\} } ( dd^c \varphi_{\mathcal J } )^n = \text { mult}_0 ( F ).$$
Moreover, by Theorem 1 in section 5.2 in \cite{AGV85}, we have
$$\int\limits_{ \{ 0\} } ( dd^c \varphi_{\mathcal I } )^n = \int\limits_{ \{ 0\} } ( dd^c \varphi_{\mathcal J } )^n = \dim_{ \mathbb C } \big(\frac { \mathcal O_{ \mathbb C^n, 0} } { \mathcal J } \big)
 = \text{Length}_{ \mathcal O_{ \mathbb C^n, 0} } \big( \frac { \mathcal O_{ \mathbb C^n, 0} } { \mathcal J } \big) = e_n (\mathcal J) = e_n (\mathcal I).$$
In general case, we have
$$\aligned
\int\limits_{ \{ 0\} }& dd^c \varphi_{\mathcal I_1}\wedge ... \wedge dd^c \varphi_{\mathcal I_n}\\
&=\frac { 1 } { n! } \sum\limits_{ k = 1 }^{ n } (-1)^{ n -k }\sum\limits_{ 1 \leq i_1 < ... < i_k\leq n } \int\limits_{ \{ 0\} } ( dd^c \sum\limits_{ 1 \leq i_1 < ... < i_k\leq n } \varphi_{ \mathcal I_{i_j} } )^n\\
&=\frac { 1 } { n! } \sum\limits_{ k = 1 }^{ n } (-1)^{ n -k }\sum\limits_{ 1 \leq i_1 < ... < i_k\leq n } e_n ( \mathcal I_{ i_1 } ...\mathcal I_{ i_k } )\\
&= e (\mathcal I_1, ..., \mathcal I_n).
\endaligned$$
\end{proof}

Proposition \ref{Monge-Ampere_Hilbert_Samuel_1} infers the following consequence:

\begin{corollary}\label{Monge-Ampere_Hilbert_Samuel_2} Let $\mathcal I$ be an ideal in $\mathcal O_{ \mathbb C^n, 0}$. Then, for all $1\leq k\leq n$, we have
$$e_k ( \varphi_{\mathcal I} ) = e_k ( \mathcal I ).$$
\end{corollary}

Demailly's approximation theorem, which is based on Ohsawa-Takegoshi's $L^2$-extension theorem (see \cite{OT87}), enables the use of tools from commutative ring theory to study the singularities of plurisubharmonic functions:

\begin{proposition}\label{Demailly_approximation} Let $\varphi$ be a plurisubharmonic function on a bounded open set $\Omega\subset\mathbb{C}^n$. For every real number $j>0$, let $\mathcal{H}_{j\varphi}(\Omega)$ be the Hilbert space of holomorphic functions $f$ on $\Omega$ such that 
$$\int\limits_{\Omega} |f|^2e^{-2j\varphi} dV_{2n}<+\infty,$$ 
and let $\psi_j=\frac{1}{2j}log\sum|g_{j,i}|^2$ where $\{ g_{j,i} \}_{ i\geq 1 }$ is an orthogonal basis of $\mathcal{H}_{j\varphi}(\Omega)$. Then, for all $1\leq k\leq n$, we have  
$$\lim\limits_{m\to\infty} c_k ( \psi_j ) ( z ) = c_k ( \varphi ) ( z ),\ \forall\ z\in\Omega .$$
\end{proposition}

\begin{proof} 
This result is derived from Theorem 1.3 in \cite{Hi17} and Theorem 6.1 in \cite{HHT21}.
\end{proof}

\section{Proof of the main results and corollaries:}

We now state the following lemma, which will be needed in the proof of the first main theorem:

\begin{lemma}\label{toric_lemma}
Let $\mathcal I$ be an $\mathfrak{m}_{ \mathcal O_{ \mathbb C^n, 0} }$-primary ideal in the ring $\mathcal O_{ \mathbb C^n, 0 }$. Then there exists a function $\psi\in PSH ( \Delta^n )\cap L_{ loc }^{ \infty } ( \Delta^n\backslash \{ 0 \} )$ such that $\psi ( z ) = \psi ( z', | z_n | ),\ \forall\ z = ( z', z_n )\in \Delta^n$, $c_n ( \varphi_{\mathcal I } ) \geq c_n ( \psi )$, $c_i ( \varphi_{\mathcal I } ) = c_i ( \psi ),\ \forall\ 1\leq i\leq n-1$, and $e_j ( \varphi_{\mathcal I } ) = e_j ( \psi ),\ \forall\ 1\leq j\leq n$.
\end{lemma}
\begin{proof}
To simplify notations, we set $\varphi = \varphi_{ \mathcal I }$. For $H\in\mathbb { G }r ( n - 1, \mathbb C^n )$, the Grassmannian manifold of $(n-1)$-dimensional subspaces in $\mathbb C^n$ and $1\leq i\leq n-1$, we define
$$\mathcal { L }_i ( H ): = \{ E\in\mathbb { G }r ( i , H ): c_i ( \varphi |_{ E } ) \not = c_i ( \varphi ) \text{ or } e_i ( \varphi |_{ E } ) \not = e_i ( \varphi ) \},$$
$$\mathcal { L }_{ i, n- 1 }: = \{ H\in\mathbb { G }r ( n - 1, \mathbb C^n ): \ \mu_{ \mathbb { G }r (  i, H ) } ( \mathcal { L }_{ i } ( H ) ) > 0  \},$$
and 
$$\mathcal L: = \bigcup\limits_{ i = 1 }^n \mathcal { L }_{ i , n - 1 },$$
where $\mu_{ \mathbb { G }r ( i , H ) }$ is the Haar measure on $\mathbb { G }r ( i , H )$. We will prove that 
$$\mu_{ \mathbb { G }r ( n - 1, \mathbb C^n ) } ( \mathcal L ) = 0.$$
Indeed, By Siu's theorem (see \cite{Si74} and also Theorem 5.14 in \cite{De92}), we have:
$$e_i ( \varphi |_{ H } ) = e_i ( \varphi ),$$
for $\mu_{ \mathbb { G }r ( i , \mathbb C^n ) }$-almost everywhere on $\mathbb { G }r ( i, \mathbb C^n )$. On the other hand, from the the proof of Theorem 1.3 in \cite{Hi17}, we have
$$c_i ( \varphi |_{ E } ) = c_i ( \varphi ),$$
for $\mu_{ \mathbb { G }r ( i , \mathbb C^n ) }$-almost everywhere on $\mathbb { G }r ( i, \mathbb C^n )$. Therefore, we obtain
$$\mu_{ \mathbb { G }r ( n - 1, \mathbb C^n ) } ( \mathcal { L }_{ i, n - 1 } ) = 0,\ \forall\ 1\leq i\leq n - 1.$$ 
These equalities imply that
$$\mu_{ \mathbb { G }r ( n - 1, \mathbb C^n ) } ( \mathcal L ) = 0.$$
Now, take an element $H\in\mathbb { G }r ( n - 1, \mathbb C^n )\backslash \mathcal L$. Then, for all $1\leq i\leq n$, we have $c_i ( \varphi |_{ H } ) = c_i ( \varphi )$ and $e_i ( \varphi |_{ H } ) = e_i ( \varphi )$. Without of loss generality, we can assume that $H = \{ z_n = 0 \}$. 

Next, consider a monomial ordering $<$ on $\mathcal O_{ \mathbb C^n , 0 }$ defined as follows: $z^{ \alpha } < z^{ \beta }$ if and only if $( \alpha_n, ..., \alpha_1 )< ( \beta_n, ..., \beta_1 )$ in lexicographic order. Let $\mathcal J$ be the ideal generated by
$$\{ g( z' ) z_n^t: \exists\ h\in\mathcal O_{ \mathbb C^n , 0 }: \text{ such that } g ( z' ) z_n^t + h ( z ) z_n^{ t + 1 }\in\mathcal I \}.$$
We define
$$u = \varphi_{ \mathcal J } \in PSH ( \Delta^n )\cap L_{ loc }^{ \infty } ( \Delta^n\backslash \{ 0 \} ).$$
It is clear that $u ( z ) = u ( z' , | z_n | ),\ \forall\ z = ( z', z_n )\in \Delta^n$  and $u |_{ H } \geq \varphi |_{ H }$. We will show that $c_i (\varphi ) \geq c_i ( u )$ and $e_i (\varphi ) \leq e_i ( u )$, for all $1\leq i\leq n$. Indeed, we can choose $\{ g_j \}_{ 1\leq j\leq k }\subset\mathcal O_{ \mathbb C^{ n - 1 } , 0 }$, $\{ h_j \}_{ 1\leq j\leq k }\subset\mathcal O_{ \mathbb C^{ n } , 0 }$ and integers $\{ t_j \}_{ 1\leq j\leq k }$ such that $\mathcal J$ is generated by $\{ g_j z_n^{ t_j } \}_{ 1\leq j\leq k}$ and $\mathcal I$ is generated by $\{ g_j z_n^{ t_j } + h_j z_n^{ t_j + 1 } \}_{ 1\leq j\leq k }$. For $s\in\mathbb N$, define the plurisubharmonic function
$$u_s= \frac { 1 } { 2 } \log \sum\limits_{ j = 1 }^k | g_j z_n^{ t_j } + \frac { 1 } { s }  h_j z_n |^2.$$
Since $u_s$ and $\varphi ( z' , \frac { 1 } { s } z_n )$ are equivalent in singularity, it follows that
$$c_i ( u_s ) = c_i ( \varphi ), \text{ and } e_i ( u_s ) = e_i ( \varphi ), \ \forall\ 1\leq i\leq n,\ s\in\mathbb N.$$
On the other hand, $u_s$ converges uniformly to $u$ on $U\backslash \mathbb B (0,r)$ for all $r >0$, where $U$ is a neighbourhood of $0$. By the main theorem in \cite{DK01}, we have
$$\liminf\limits_{s\to\infty} c_i (u_s ) \geq c_i (u ),$$
and Theorem 3.4 in \cite{Hi08} , we have
$$\limsup\limits_{s\to\infty} e_i ( u_s ) \leq e_i (u ).$$
Therefore, we conclude that
$$c_i (\varphi ) \geq c_i ( u ) \text{ and } e_i (\varphi ) \leq e_i ( u ),\ \forall 1\leq i\leq n.$$ 
For $1\leq i \leq n-1$, since $u |_{ H } \geq \varphi |_{ H }$, it follows that
$$c_i (\varphi ) = c_i ( \varphi |_{ H } ) \leq c_i ( u |_{ H } ) \leq c_i ( u ),$$
and 
$$e_i (\varphi ) = e_i ( \varphi |_{ H } ) \geq e_i ( u |_{ H } ) \geq e_i ( u ).$$
Combining the above inequalities, we obtain 
$$c_i (\varphi ) = c_i ( u )\ \text{ and } e_i (\varphi ) = e_i ( u ),\ \forall\ 1\leq i\leq n - 1.$$ 
We denote by $\mathcal {J}_p$ the ideal generated by 
$$\{ g( z' ) z_n^t: \exists\ h\in\mathcal O_{ \mathbb C^n , 0 }: \text{ such that } g ( z' ) z_n^t + h ( z ) z_n^{ t + 1 }\in\mathcal {I}^p\}.$$
Since $\mathcal {J}_p^q\subset \mathcal {J}_{ p q }$, for all $p,q\in\mathbb N$, it follows that 
$$\varphi_{ \mathcal J_1 }\leq \frac { 1 } { 2} \varphi_{ \mathcal {J}_{ 2 } }\leq ...$$
Define
$$\psi = ( \limsup\limits_{ p\to\infty } \frac { 1 } { 2^p } \varphi_{ \mathcal J_{ 2^p } } )^{ * }.$$
We observed that $\psi\in PSH ( \Delta^n )\cap L_{ loc }^{ \infty } ( \Delta^n\backslash \{ 0 \} )$, $\psi ( z ) = \psi ( z' , | z_n | ),\ \forall\ z = ( z' , z_n )\in \Delta^n$ and $\frac { 1 } { 2^p } \varphi_{ \mathcal J_{ 2^p } }\nearrow\psi$. Therefore, we have
$$\aligned 
c_i ( \psi ) &= \lim\limits_{ p\to\infty } 2^p  c_i ( \varphi_{\mathcal J_{ 2^p } } )\\
&= \lim\limits_{ p\to\infty } 2^p  c_i ( \varphi_{ \mathcal {I}^{ 2^p } } )\\
&= c_i ( \varphi ),\ \forall\ 1\leq i\leq n - 1,
\endaligned$$
$$\aligned 
c_n ( \psi ) &= \lim\limits_{ p\to\infty } 2^p  c_n( \varphi_{\mathcal {J}_{ 2^p } } )\\
&\geq \lim\limits_{ p\to\infty } 2^p  c_n ( \varphi_{ \mathcal {I}^{ 2^p } } )\\
&= c_n ( \varphi ),
\endaligned$$
$$\aligned 
e_i ( \psi ) &= \lim\limits_{ p\to\infty } \frac { 1 } { 2^{ ip } }  e_i ( \varphi_{\mathcal {J}_{ 2^p } } )\\
&= \lim\limits_{ p\to\infty } \frac { 1 } { 2^{ ip } } e_i ( \varphi_{ \mathcal {I}^{ 2^p } } )\\
&= e_i ( \varphi ),\ \forall\ 1\leq i\leq n - 1,
\endaligned$$
and 
$$\aligned 
e_n ( \psi ) &= \lim\limits_{ p\to\infty } \frac { 1 } { 2^{ np } }  e_n ( \varphi_{\mathcal {J}_{ 2^p } } )\\
&\geq\lim\limits_{ p\to\infty } \frac { 1 } { 2^{ np } } e_n ( \varphi_{ \mathcal {I}^{ 2^p } } )\\
&= e_n ( \varphi ),
\endaligned$$
It remains to prove that
$$e_n ( \psi )\leq e_n ( \varphi ).$$
Indeed, we observed that $IM ( \mathcal I^{ p } ) = IM ( \mathcal J_{ p } )$ for all $p\geq 1$. Therefore, by Proposition \ref{Monge-Ampere_Hilbert_Samuel_1}, Proposition \ref{dimension_ideal} and Lech's theorem (see Theorem 3.1 in \cite{Le60}) we obtain 
$$\aligned
e_n (\psi ) &= \lim\limits_{ p\to\infty } \frac { 1 } { 2^{ np } } e_n ( \varphi_{ \mathcal {J}_{ 2^p } } )\\
&=\lim\limits_{ p\to\infty } \frac { 1 } { 2^{ np } } e_n ( \mathcal {J}_{ 2^p } )\\
&\leq\lim\limits_{ p\to\infty } \frac { 1 } { 2^{ np } } n! \text{Length}_{ \mathcal{O}_{ \mathbb C^n, 0 } } \big( \frac { \mathcal{O}_{ \mathbb C^n, 0 } } { \mathcal {J}_{ 2^p } } \big)\\
&=\lim\limits_{ p\to\infty } \frac { 1 } { 2^{ np } } n! \dim_{ \mathbb C } \big( \frac { \mathcal{O}_{ \mathbb C^n, 0 } } { \mathcal {I}^{ 2^p } } \big)\\
&=e_n ( \mathcal {I} ) = e_n ( \varphi ).
\endaligned$$
\end{proof}

{\bf 5.1. Proof of Theorem \ref{Theorem1}:}

We will prove Theorem \ref{Theorem1} by induction on the dimension $n$. Clearly, the statement holds for $n=1$. Assume the theorem holds for dimension $n-1$. As arguments in Lemma \ref{toric_lemma}, we can choose a complex hyperplane $H\subset\mathbb C^n$ through $0$ such that for all $1\leq k\leq n - 1$,
$$c_k(\varphi ) = c_k(\varphi |_{ H }),$$
and 
$$e_k(\varphi ) = e_k(\varphi |_{ H }).$$
By the induction hypothesis (which applies to $n-1$), for all $1\leq k\leq n-1$, we have 
$$c_k (\varphi) \geq c_{ k - 1} (\varphi ) + \frac { 1 } { \prod\limits_{ i = 1 }^{ k - 1 } \big( c_{ i } ( \varphi ) - c_{ i - 1 } ( \varphi ) \big) e_k (\varphi ) },$$
Now, we only need to prove
$$c_n (\varphi) \geq c_{ n - 1} (\varphi ) + \frac { 1 } { \prod\limits_{ i = 1 }^{ n - 1 } \big( c_{ i } ( \varphi ) - c_{ i - 1 } ( \varphi ) \big) e_n (\varphi ) },$$
 Indeed, we will prove the result in three steps:

Step 1: We first prove the theorem in the case where
$$\varphi ( z ) = \max\limits_{ 0\leq i\leq m + 1 } \{ u_i + a_i \log | z_n | \},$$ 
with 
$$0 = a_0 < a_1 < ... < a_m < a_{ m + 1 } = 1,$$ 
and $u_0, u_1, ..., u_{m+1}\in\mathcal E ( \Delta^{ n - 1 } )$, $u_0\leq u_1\leq ... \leq u_m\leq u_{ m + 1 } = 0$. For each $t\in ( a_i ,a_{ i + 1 } ]$ ($1\leq i\leq m$ ), we define
$$v_t: = \max \{ u_0 , \frac { t } { t - a_1 } u_1,..., \frac { t } { t - a_i } u_i \} \in\mathcal E ( \Delta^{ n - 1 }),$$
and
$$\psi_t ( z ) = \max \{ v_t ( z' ), t \log | z_n | \}.$$
We observe that
$$\varphi = \min\{ \psi_t:\ t\in (a_1 , 1 ] \}.$$
We aim to prove that 
$$c_{ \mu } (\varphi ) = \inf\{ c_{ \mu } ( \psi_t ) :\  t\in (a_1 , 1 ]\},$$
for all locally moderate measure $\mu$ (see \cite{DNS10} for the definition). It suffices to show that
$$c_{ \mu } (\varphi ) \geq \inf\{ c_{ \mu } ( \psi_t ):\ t\in (a_1 , 1 ]\}.$$
Indeed, we define
$$\varphi_{ j }: = \min\{ \psi_t :\ t\in \{ a_1, a_1 + \frac { 1 - a_1 } { j }, ..., 1 \} \}.$$
Since 
$$\psi _{ t_2 } \geq \psi _{ t_1 } + [ t_2 - t_1 ]  \log | z_n |,\ \forall a_1\leq t_1\leq t_2\leq 1,$$ 
we obtain
$$\varphi_{ j } + \frac { 1 - a_1 } { j } \log | z_n | \leq \varphi.$$
Hence 
$$c_{ \mu } ( \varphi_{ j } + \frac { 1 - a_1 } { j } \log | z_n | )  \leq c_{ \mu } ( \varphi ).$$
Moreover, by Proposition \ref{weighted_lct}, we get
$$\aligned
c_{\mu } ( \varphi ) &\geq \lim\limits_{ j\to \infty } c_{ \mu } ( \varphi_{ j } )\\
&=\min\{ c_{ \mu } ( \psi_t ):\ t\in \{ a_1, a_1 + \frac { 1 - a_1 } { j }, ..., 1 \} \}\\
&\geq \inf\{ c_{ \mu } ( \psi_t ):\ t\in (a_1 , 1 ]\}.
\endaligned$$
Define the function
$$F_n (c_1,...,c_n) = c_1( c_2-c_1 ) ... ( c_{n} - c_{n-1} ).$$
Our goal is to show:
$$F_n ( c_1 ( \varphi ), ... , c_{ n } ( \varphi  ) )\geq e_n ( \varphi ) ).$$
Indeed, we can choose $t_0\in (a_1 , 1 ]$ such that
$$\aligned 
c_n ( \varphi ) &= c_n ( \psi_{ t_0 } )\\
&= c_{ n - 1 } ( v_{ t_0 } ) + \frac { 1 } { t_0 }\\
&= c_{ n - 1 } ( v_{ t_0 } ) + \frac { e_{ n - 1 } ( v_{ t_0 } ) } { e_n ( \psi_{ t_0 } ) }
\endaligned$$
Define $d_n = c_n ( \psi_{ t_0 } ) = c_n ( \varphi )$, $d_1 = c_1 ( v_{ t_0 } )$ and for $2\leq i\leq n-1$:
$$d_{i} = \frac { c_{ i - 1 } ( v_{ t_0 } ) + d_{ i + 1 } }  { 2 } + | c_{ i } ( v_{ t_0 } ) -  \frac { c_{ i - 1 } ( v_{ t_0 } ) + d_{ i + 1 } }  { 2 } |.$$
We then observe:
$$d_i\geq c_i ( v_{ t_0 } )\geq c_i (u_0) = c_i ( \varphi ),\ \forall\ 1\leq i\leq n-1,$$
$$d_{ i + 1 } - d_{ i } =\min ( d_{ i + 1 } - c_i ( v_{ t_0 } ), c_i ( v_{ t_0 } ) - c_{ i - 1 } ( v_{ t_0 } ) ) \leq d_{ i  } - d_{ i - 1 },\ \forall\ 1\leq i\leq n-1,$$
and
$$F_n ( c_1 ( v_{ t_0 } ), ... , c_{ n } ( \psi_{ t_0 } ) ) = F_n ( d_1, ..., d_n ),$$
By Propositions \ref{concave_property} and \ref{weighted_property}, both $( c_1 ( \varphi ), ... , c_{ n } ( \varphi  ) )$, $( d_1 , ... , d_n )$ lie in the convex cone
$$D = \{ c\in ( \mathbb R^{ + } )^n:\ c_{ i + 1 } - c_{ i } \leq c_{ i } - c_{ i - 1 },\ \forall\ 1\leq i\leq n-1 \}.$$
Moreover, since $\frac { \partial F_n } { \partial c_i }\leq 0$ on $D$ for all $1\leq i\leq n - 1$, we conclude: 
$$F_n ( c_1 ( \varphi ), ... , c_{ n } ( \varphi  ) ) \geq F_n ( d_1, ..., d_n ) = F_n ( c_1 ( v_{ t_0 } ), ... , c_{ n } ( \psi_{ t_0 } ) ),$$
By the induction hypothesis (which applies to $n-1$) and $\varphi\leq\psi_{ t_0 }$, we have
$$F_n ( c_1 ( v_{ t_0 } ), ... , c_{ n } ( \psi_{ t_0 } ) ) = \frac { 1 } { t_0 } F_{ n - 1 } ( c_1 ( v_{ t_0 } ), ... , c_{ n - 1 } ( v_{ t_0 } ) ) \geq \frac { 1 } { t_0 e_{ n - 1 } ( v_{ t_0 } ) }$$
$$= \frac { 1 } {  e_{ n } ( \psi_{ t_0 }) } \geq \frac { 1 } { e_n (\varphi ) }.$$
Combining these inequalities, we obtain
$$F_n ( c_1 ( \varphi ), ... , c_{ n } ( \varphi  ) ) \geq \frac { 1 } { e_n (\varphi ) }.$$
Step 2: We now consider the case $\varphi ( z ) = \varphi_{ \mathcal I }$, where $\mathcal I\subset\mathcal O_{ \mathbb C^n, 0 }$ is an $\mathfrak{m}_{ \mathcal O_{ \mathbb C^n, 0} }$-primary ideal.

By Lemma \ref{toric_lemma}, there exists a function $\psi\in PSH ( \Delta^n )\cap L_{ loc }^{ \infty } ( \Delta^n\backslash \{ 0 \} )$ such that $\psi ( z ) = \psi ( z' , | z_n | ),\ \forall\ z = ( z', z_n )\in \Delta^n$, $c_n ( \varphi_{\mathcal I } ) \geq c_n ( \psi )$, $c_i ( \varphi_{\mathcal I } ) = c_i ( \psi ),\ \forall\ 1\leq i\leq n-1$, and $e_j ( \varphi_{\mathcal I } ) = e_j ( \psi ),\ \forall\ 1\leq j\leq n$.
From the construction in Lemma \ref{toric_lemma}, we can assume that there exist plurisubharmonic functions $u_0, ... , u_s$ in the class $\mathcal E(\Delta^{n-1})$ and constants $0\leq a_1 \leq ... \leq a_s$ such that
$$\psi ( z ) = \max\limits_{ 0\leq i \leq s + 1 }  \{ u_i + a_i \log | z_n |  \},$$
where $a_0 = 0$ and $u_{ s + 1 } = 0$. Choose integers $0 = k_0 < k_1 < k_2 < ... < k_{ m + 1 }$ such that
$$0 = a_0 = ... = a_{ k_1 - 1 } < a_{ k_1 } = ... =a_{ k_2 - 1 } < ...< a_{ k_{ m + 1 } } = ... = a_{ s },$$
We set
$$v_j = \max \{ u_i : 0\leq i\leq k_{ j + 1 } - 1 \},\ \forall 1\leq j\leq m,$$
$v_{m+1}=0$ and and $b_j = a_{ k_j }$, $\ \forall\ 0\leq j\leq m + 1$. We observed that
$$\psi = \max \{ v_j + b_j \log| z_n | : \ 0\leq j\leq m + 1 \}.$$
By step 1, we conclude:
$$F_n ( c_1 ( \varphi ), ... , c_{ n } ( \varphi  ) )  \geq F_n ( c_1 ( \psi ), ... , c_{ n } ( \psi  ) ) \geq \frac { 1 } { e_n ( \psi ) } = \frac { 1 } { e_n (\varphi ) }.$$
Step 3:Finally, we consider the general case.

For every real number $j>0$, let $\mathcal{H}_{j\varphi}(\Omega)$ be the Hilbert space of holomorphic functions $f$ on $\Omega$ such that 
$$\int\limits_{\Omega} |f|^2e^{-2j\varphi} dV_{2n}<+\infty .$$ 
Let 
$$\psi_j=\frac{1}{2j}log\sum|g_{j,i}|^2,$$ 
where $\{ g_{j,i} \}_{ i\geq 1 }$ is an orthogonal basis of $\mathcal{H}_{j\varphi}(\Omega)$. Then, by Step 2 and Proposition \ref{Demailly_approximation}, we obtain:
$$F_n ( c_1 ( \varphi ), ... , c_{ n } ( \varphi  ) ) = \lim\limits_{ j\to \infty } F_n ( c_1 ( \psi_j ), ... , c_{ n } ( \psi_j ) ) \geq  \lim\limits_{ j\to \infty }  \frac { 1 } { e_n ( \psi_j ) } \geq \frac { 1 } { e_n (\varphi ) }.$$
\qed

{\bf 5.2. Proof of Theorem \ref{Theorem2}:}

We begin by noting the obvious inequality:
$$c_{ || z ||^{2t} dV_{2n} } ( \varphi_{ f } ) \leq \min \left\{ c_{ || z ||^{2t} dV_{2n} } ( \varphi_{ \bar { \mathcal  D } ( f ) } ), 1 \right\},\ \forall\ t > -n,$$
It remains to prove the reverse inequality:
$$c_{ || z ||^{2t} dV_{2n} } ( \varphi_{ f } ) \geq \min \left\{ c_{ || z ||^{2t} dV_{2n} } ( \varphi_{ \bar { \mathcal  D } ( f ) } ), 1 \right\},\ \forall\ t > -n,$$
Let $\{ f_i \}_{1\leq i\leq m}$ be a basis of the ideal $\bar{ \mathcal {D} } ( f )$. Define a holomorphic function $F:\Delta_{ r }^{ n + m }\to \mathbb C$ by:
$$F( z , w ) = f( z ) + \sum\limits_{ 1\leq i\leq m} w_i f_i .$$
Let $\mathcal I\subset\mathcal O_{ \mathbb C^{ n + m }, 0}$ be the ideal generated by $\{F, z_i \frac { \partial F } { \partial z_j } : 1\leq i ,j\leq n\}$. Since $\mathcal D ( \overline{ \mathcal D } ( f ) ) = \mathcal D (f)$, there exist $a_{ ijkpq }, a_{ ijk }\in\mathcal { O }_{ \mathbb C^{n}, 0 }$ such that 
$$z_j \frac { \partial f_i } { \partial z_k } = \sum\limits_{ 1\leq p, q\leq n } a_{ijkpq} z_p \frac { \partial f} { \partial z_q} + a_{ ijk } f.$$ 
Therefore
$$\aligned z_j \frac { \partial F } { \partial z_k } &= z_j \frac { \partial f } { \partial z_k } + \sum\limits_{ 1\leq i \leq m } w_i z_j \frac { \partial f_i } { \partial z_k }\\ 
&= z_j \frac { \partial f } { \partial z_k } + \sum\limits_{ 1\leq i \leq m } w_i \sum\limits_{ 1\leq p, q\leq n } [ a_{ ijkpq } z_p \frac { \partial f} { \partial z_q} + a_{ ijk } f ]. 
\endaligned$$ 
Hence,
$$\sum\limits_{ 1\leq p, q\leq n } [ \delta_{ (j,k) (p,q) } + \sum\limits_{ 1\leq i\leq m} w_i a_{ ijkpq } ] z_p \frac { \partial f} { \partial z_q}  = z_j \frac { \partial F } { \partial z_k } - f \sum\limits_{ 1\leq i \leq m } w_i a_{ ijk }.$$ 
Since the matrix $[ \delta_{ (j,k) (p,q) } + \sum\limits_{ 1\leq i\leq m} w_i a_{ ijkpq } ]_{1 \leq j,k,p,q\leq n}$, with coefficients in $\mathcal { O }_{ \mathbb C^{ n + m }, 0 }$, is invertible, there exist $b_{ pqjk },  b_{ pq }, h_{ ijk }, h_i\in\mathcal { O }_{ \mathbb C^{ n + m }, 0 }$ such that 
$$z_p \frac { \partial f} { \partial z_q} = \sum\limits_{ 1 \leq j,k \leq n } b_{ pqjk } z_j \frac { \partial F } { \partial z_k } + b_{ pq } f,$$ and 
$$f_i = \sum\limits_{ 1 \leq j,k \leq n } h_{ ijk } z_j \frac { \partial F } { \partial z_k } + h_{ i } f,$$ 
for all $1\leq p,q \leq n$, $1\leq i\leq m$. Thus, we find 
$$\aligned F &= f + \sum\limits_{ 1\leq i\leq m } w_i f_i\\ 
&= f + \sum\limits_{ 1\leq i\leq m } w_i \sum\limits_{ 1 \leq j,k \leq n } [h_{ ijk } z_j \frac { \partial F } { \partial z_k } + h_{ i } f]\\ 
&=f [ 1 + \sum\limits_{ 1 \leq i\leq m } w_i h_{ i } ] + \sum\limits_{ 1\leq i\leq m } w_i \sum\limits_{ 1 \leq j,k \leq n } h_{ ijk } z_j \frac { \partial F } { \partial z_k }. \endaligned$$ 
This shows that $f\in \mathcal I$. Hence, 
$$\frac { \partial F } { \partial w_{ i } } = f_i\in\mathcal I,\ \forall 1\leq i\leq m.$$ 
By Theorem 9.1.7 in \cite{JP00}, there exist holomorphic maps $\Phi : \Delta_r^{ n + m }\to \mathbb C^n$ and $h: \Delta_r^{ n + m }\to \mathbb C$, such that $\Phi ( z , 0 ) = z $, $\Phi ( 0 , w ) = 0$, $h ( z , 0 ) = 1$ 
$$\det [ \frac { \partial \Phi_i } { \partial z_j } (0,0) ]_{ 1\leq i, j\leq n } = 1,$$
and
$$F ( z , w ) = h ( z ,w ) F ( \Phi ( z , w ) , 0 ) = h ( z ,w ) f ( \Phi ( z , w ) ).$$
From this, it follows that there exist $\epsilon > 0$ such that
$$c_{ ||z||^{ 2t } dV_{ 2n } } ( \varphi_{ f } ) = c_{ ||z||^{ 2t } dV_{ 2n } } ( \log | F ( . , w ) | ),\ \forall\ w\in \Delta_{ \epsilon }^{ m },\ \forall\ t > -n.$$
On the other hand, by general version of Proposition 1.5 in \cite{DK01} (see Proposition 2.2 in \cite{HHT21}), we have
$$c_{ ||z||^{ 2t } dV_{ 2n } } ( \log | F ( . , w )| ) = \min \left\{ c_{ ||z||^{ 2t } dV_{ 2n } } (  \varphi_{ \bar{ \mathcal {D} } ( f ) } ) , 1 \right\}$$
for $dV_{ 2m }$-almost everywhere $w\in\mathbb C^{ m }$. Hence, we conclude:
$$c_{ ||z||^{ 2t } dV_{ 2n } } ( \varphi_{ f }) = \min \left\{ c_{ ||z||^{ 2t } dV_{ 2n } } ( \varphi_{ \bar{ \mathcal {D} } ( f ) } ) , 1  \right\},\ \forall\ t > -n.$$

\qed

{\bf 5.3. Proof of Corollary \ref{Corollary1}:}

The statement is clearly true if the order of vanishing of $f$ at $0$ is less than or equal to 1. In the case where the order is greater than 1, observe that $\frac { \partial f } { \partial z_i }\frac { \partial f } { \partial z_j }\in \bar D ( f )$ for all $1\leq i, j\leq n$, It follows that:
$$\bigcap\limits_{ g \in \bar { \mathcal  D } ( f )  } \{ g = 0 \} = \{f = \frac { \partial f } { \partial z_1 } = ... = \frac { \partial f } { \partial z_n }= 0\}.$$ 
Thus, the inequality in Corollary \ref{Corollary1} follows directly from Theorem \ref{Theorem1} and  \ref{Theorem2}.
\qed

{\bf 5.4. Proof of Corollary \ref{Corollary2}:}

By Theorem \ref{Theorem2}, we have
$$c_n(\varphi_f) = \min \left\{ c_n(\varphi_{\bar{D}(f)}), 1 \right\}.$$
Moreover, since \( D\left( \langle z_1^{m+1}, \dots, z_n^{m+1} \rangle \right) \subset D(f) \), it follows that
$$\langle z_1^{m+1}, \dots, z_n^{m+1} \rangle \subset \bar{D}(f).$$
This implies that $\varphi_{\bar{D}(f)} \geq (m+1) \log \|z\|$. By Proposition \ref{ltc_inequality}, we obtain
$$c_n(\varphi_{\bar{D}(f)}) \geq c_{n-1}(\varphi_{\bar{D}(f)}) + \frac{1}{m+1}.$$
Therefore,
$$c_n(\varphi_f) = \min \left\{ c_n(\varphi_{\bar{D}(f)}), 1 \right\} \geq \min \left\{ c_{n-1}(\varphi_{\bar{D}(f)}) + \frac{1}{m+1}, 1 \right\}\geq\min \left\{ c_{n-1}(\varphi_{ f } ) + \frac{1}{m+1}, 1 \right\}.$$
\qed

\section{Singularity invariants of complex spaces} 

A complex space $X$ is a topological space equipped with an atlas of charts, each isomorphic to an analytic subset of a complex Euclidean space, such that the transition maps are holomorphic in the sense of holomorphic mappings between analytic sets. Complex spaces, particularly subvarieties of complex manifolds, play a central role in both complex and algebraic geometry.

In this section, we explore how singularity invariants associated with plurisubharmonic functions can be extended to complex spaces. These invariants have the potential to serve as tools in the classification and study of complex spaces.

{\bf 6.1. Singularity invariants of plurisubharmonic functions:}

First, let $inv : PSH_{ 0 } \to [ -\infty, +\infty ]$ be a function. We say that $inv$ is a singularity invariant if 
$$inv ( \varphi_{\circ}  H ) = inv (\varphi ),$$ 
for every biholomorphic germ $H: ( \mathbb C^n, 0 ) \to ( \mathbb C^n, 0 )$. 
Given such a singularity invariant $inv$ and a real number $a>0$, we define
$$inv_a ( \varphi ) = inv ( \max (\varphi , a \log || z ||) ).$$
Then $\{ inv_a \}_{ a>0 }$ forms a family of singularity invariants of plurisubharmonic functions. 

{\bf 6.2. Singularity invariants of complex spaces:}

Let $X$ be a complex space, and let $\mathcal{O}_{X, x} $ denote the ring of germs of holomorphic functions at a point $x\in X$. The maximal ideal $\mathfrak{m}_{X,x}$ of the local ring $\mathcal{O}_{X, x} $ is defined as follows:
$$\mathfrak{m}_{X,x} := \{ f \in \mathcal{O}_{X, x} : f(x) = 0 \}.$$
For each point $x\in X$, we define the embedding dimension of $X$ at $x$ by: 
$$N(X,x) = \dim_{ \mathbb C } \big(  \frac {  \mathfrak{m}_{X,x} } {  \mathfrak{m}_{X,x}^2 } \big).$$
The embedding dimension $N(X,x)$ is the smallest integer $N$ such that $(X,x)$ can be embedded in $\mathbb C^N$ (see Section 9.8 of Chapter II in \cite{De12}). Accordingly, we choose a neighborhood $A$ of $x\in X$ such that $(A,x)$ is biholomorphic to an analytic subset $(B,0)$ of a domain $\Omega \subset \mathbb{C}^{N(X,x)} $. We denote $\mathcal{I}(B, \Omega)_0$ as the ideal of germs of holomorphic functions at $0 \in \Omega$ that vanish on $B$. Let $inv : PSH_0 \to [-\infty, +\infty]$ be a singularity invariant of plurisubharmonic functions. By Lemma 4.1 in \cite{Wh65}, this number $inv( \varphi_{\mathcal{I}(B, \Omega)_0} )$ depends only on $(X,x)$ and  is independent of the choice of the pair $(B, \Omega )$. Thus, we define the singularity invariant of $X$ at $x$ as follows:
$$inv(X, x) := inv( \varphi_{\mathcal{I}(B, \Omega)_0} ).$$
We also define the global singularity invariants of the complex space $X$ as follows:
$$inv^{min} (X) = \inf \{ inv (X, x):\ x\in X_{sing} \},$$ 
$$inv^{max} (X) = \sup \{ inv (X, x):\ x\in X_{sing} \},$$ 

\begin{theorem}\label{Theorem3} 
Let $X$ and $Y$ be complex spaces such that there exists a biholomorphic map $f:X\to Y$. Then 
$$inv ( X, x ) = inv ( Y, f ( x ) ).$$
\end{theorem}

\begin{proof}
We choose a neighborhood $A$ of $x\in X$ (resp. a neighborhood $P$ of $f(x)\in Y$) such that $(A,x)$ is biholomorphic to an analytic subset $(B,0)$ of a domain $\Omega \subset \mathbb{C}^{N(X,x)}$ (resp. $(P, f(x) )$ is biholomorphic to an analytic subset $(Q,0)$ of a domain $D \subset \mathbb{C}^{N(Y, f(x) )}$). Then there exists a biholomorphic germ 
$$g: (B,0) \to (Q,0).$$ 
Since $N(X,x) = N(Y, f ( x ) )$, by Lemma 4.1 in \cite{Wh65}, the germ $g$ can be extended to a biholomorphic germ 
$$h: (\Omega , 0)\to (D, 0).$$ 
Hence,
$$\varphi_{\mathcal{I}(B, \Omega)_0} = { \varphi_{\mathcal{I}(Q, D)_0} }_\circ h.$$
Therefore,
$$inv ( X, x ) = inv( \varphi_{\mathcal{I}(B, \Omega)_0} ) = inv( \varphi_{\mathcal{I}(Q, D)_0} ) = inv ( Y, f ( x ) ).$$
\end{proof}

\vskip 0.5cm
\n

{\bf Acknowledgment:} This research is funded by Vietnam National Foundation for Science and Technology Development (NAFOSTED) under grant number 101.02-2021.16.

\vskip 0.5cm
\n

{\bf Data availability} No data was analyzed in this paper.

\vskip 0.5cm
\n

{\bf \Large Declarations}

\vskip 0.5cm
\n

{\bf Conflict of Interest} On behalf of all authors, the corresponding author states that there is no conflict of interest.

\vskip 0.5cm
\n
Pham Hoang Hiep

\n
ICRTM, Institute of Mathematics, Vietnam Academy of Science and Technology, 

\n
18, Hoang Quoc Viet, Hanoi, Vietnam

\n
{\it e-mail\/}: {\tt phhiep@math.ac.vn}

\end{document}